\documentclass[11pt]{article}
\usepackage{amsmath}
\usepackage{amsthm}
\usepackage{amssymb}
\usepackage{amscd}
\usepackage{relsize}
\usepackage{enumerate}

\newtheorem{teo}{Theorem}
\newtheorem{prop}{Proposition}
\newtheorem{cor}[teo]{Corollary}

\newcommand{\sinc}{\operatorname{sinc}}
\newcommand{\espan}{\operatorname{span}}
\DeclareMathOperator*{\esup}{ess\,sup}
\DeclareMathOperator*{\einf}{ess\,inf}

\title{{\bf Convolution systems on discrete abelian groups as a unifying strategy in sampling theory}}
\author{
{\bf A.~G. Garc\'{\i}a}\thanks{E-mail:\texttt{agarcia@math.uc3m.es}}, \,\,
{\bf M.~A. Hern\'andez-Medina}\thanks{E-mail:\texttt{miguelangel.hernandez.medina@upm.es}}
{\bf\,\, and \,\, G. P\'erez-Villal\'on}\thanks{E-mail:\texttt{gerardo.perez@upm.es}}
}
\date{}
\begin{document}
\maketitle
\begin{itemize}
\item[*] Departamento de Matem\'aticas, Universidad Carlos III de Madrid,
 Avda. de la Universidad 30, 28911 Legan\'es-Madrid, Spain.
\item[\dag] Information Processing and Telecommunications Center, Universidad Polit\'ecnica de Madrid, Departamento de Matem\'atica Aplicada a las Tecnolog\'{\i}as de la Informaci\'on y las Comunicaciones, E.T.S.I.T., Avda. Complutense 30, 28040 Madrid, Spain.
 \item[\ddag] Departamento de Matem\'atica Aplicada a las Tecnolog\'{\i}as de la Informaci\'on y las Comunicaciones, E.T.S.I.T., Universidad Polit\'ecnica de Madrid,
 Nicola Tesla s/n, 28031 Madrid, Spain.
\end{itemize}
\begin{abstract}
A regular sampling theory in a multiply generated unitary invariant subspace of a separable Hilbert space $\mathcal{H}$ is proposed. This subspace is associated to a unitary representation of a countable discrete abelian group $G$ on $\mathcal{H}$. The samples are defined by means of a filtering process which generalizes the usual sampling settings. The multiply generated  setting allows to consider some examples where the group $G$ is non-abelian as, for instance, crystallographic groups. Finally, it is worth to mention that classical average or pointwise sampling in shift-invariant subspaces are particular examples included in the followed approach.
\end{abstract}

\noindent{\bf AMS}: 42C15; 94A20; 22B05; 20H15.
\section{Introduction}
\label{section1}
In this paper we propose a regular sampling theory for a multiply generated $U$-invariant subspace of a separable Hilbert space $\mathcal{H}$. By regular sampling we mean that the samples are taken following the pattern given by the action of a discrete abelian group $G$ on $\mathcal{H}$ by means of a unitary representation $g\mapsto U(g)$ of the group $G$ on $\mathcal{H}$ which also defines the $U$-invariant subspace where the sampling will be carried out. Recall that a unitary representation of $G$ on $\mathcal{H}$ is a homomorphism of $G$ into  the group of unitary operators in $\mathcal{H}$.

In classical shift-invariant subspaces of $L^2(\mathbb{R}^d)$ this group is $\mathbb{Z}^d$ or a subgroup of it, and the unitary representation is given by the integer shifts. In general, the $U$-invariant subspace in $\mathcal{H}$ looks like
\[
\mathcal{V}_{\Phi}=\Big\{ \sum_{n=1}^N\sum_{g\in G} x_{n}(g) U(g)\varphi_{n} \, :\,  x_{n}\in \ell^2(G),\,\, n=1, 2, \dots, N\Big\}\,,
\]
where $\Phi=\{\varphi_{1},\varphi_{2,}\ldots,\varphi_{N}\}$ denotes a fixed set of generators in $\mathcal{H}$.

For each $f\in \mathcal{V}_\Phi$ we consider two sets of samples $\{\mathcal{L}_mf(g)\}_{g\in G;\,n=1, 2,\dots,M}$  defined as  
\[
\mathcal{L}_{m}f(g):=\big\langle f, U(g)\psi_{m} \big\rangle_{\mathcal{H}}\quad \text{or}\quad \mathcal{L}_{m}f(g):=\big[U(-g)f \big](t_m),\quad g\in G\,,
\]
where, in the first case $\psi_1, \psi_2, \dots, \psi_M$ denote $M$ elements in $\mathcal{H}$, which do not belong necessarily to 
$\mathcal{V}_\Phi$, and, in the second case, we take $\mathcal{H}:=L^2(\mathbb{R}^d)$ and $t_1, t_2, \dots, t_M$ are $M$ fixed points in $\mathbb{R}^d$. In the special case where $\mathcal{H}:=L^2(\mathbb{R}^d)$, $G:=\mathbb{Z}^d$ and $\big[U(p)f\big](t):=f(t-p)$, $t\in \mathbb{R}^d$ and $p\in \mathbb{Z}^d$, the above samples correspond to average or pointwise sampling, respectively, in the corresponding shift-invariant subspace $V_\Phi^2$ of $L^2(\mathbb{R}^d)$.

\medskip

These data samples have in common that can be expressed as a convolution system in the  product Hilbert space $\ell^2_{_N}(G):=\ell^2(G)\times \dots \times \ell^2(G)$ ($N$ times); namely, for $f= \sum_{n=1}^N\sum_{g\in G} x_{n}(g) U(g)\varphi_{n}$ in $\mathcal{V}_\Phi$ and $m=1,2, \dots,M$ we have
\[
\mathcal{L}_{m}f(g)=\sum_{n=1}^N (a_{m,n}\ast x_{n})(g),\,\,\, g\in G\,,
\] 
for some $MN$ sequences $a_{m,n}\in \ell^2(G)$ (see Section \ref{section2} below). Thus, in general, the acquisition of samples can be modeled as a filtering process in $\ell^2_{_N}(G)$. This is very usual situation: the regular samples 
$\{f(n)\}_{n\in \mathbb{Z}}$ of any bandlimited function $f$ in the Paley-Wiener space $PW_\pi$ are given as $f(n)=(f\ast \sinc)(n)$, \, $n\in \mathbb{Z}$, where $\sinc$ denotes the cardinal sine function. In the case of average sampling we have that $\langle f, \psi(\cdot-n)\rangle_{L^2(\mathbb{R})}=(f\ast \widetilde{\psi})(n)$, \, $n\in \mathbb{Z}$, where 
$\widetilde{\psi}(t)=\overline{\psi(-t)}$ is the average function.

\medskip

Under appropriate hypotheses on the Fourier transforms $\widehat{a}_{m,n}\in L^2(\widehat{G})$ of $a_{m,n}\in \ell^2(G)$ we obtain (see Thm.~\ref{Usampling} in Section \ref{section4}) necessary and sufficient conditions for the existence of stable reconstruction formulas in $\mathcal{V}_\Phi$ having the form
\[
f=\sum_{m=1}^M\sum_{g\in G} \mathcal{L}_m f(g)\,U(g)S_m\,,
\]
for some sampling functions $S_m \in \mathcal{V}_\Phi$,\, $m=1,2,\dots,M$, where the corresponding sequence $\{U(g)S_m\}_{g\in G;\, m=1,2,\dots,M}$ forms a frame for 
$\mathcal{V}_\Phi$. The use of the Fourier transform in $\ell^2(G)$, and the discrete nature of the sampling problem treated here impose that $G$ will be a countable discrete abelian group. However, as we will see in Section \ref{subsection4-2}, some cases involving non-abelian groups expressed as semi-direct or direct product of groups can be considered inside our study; this is the case of crystallographic groups. Notice that working in  locally compact abelian groups is not just a unified way of dealing with the classical  groups 
$\mathbb{R}^d, \mathbb{Z}^d, \mathbb{T}^d, \mathbb{Z}_s^d$: signal processing often involves products of these groups which are also locally compact abelian groups. For example, multichannel video signal involves the group $\mathbb{Z}^d \times \mathbb{Z}_s$, where $d$ is the number of channels and $s$ the number of pixels of each image.

\medskip

The used mathematical technique is that of frame theory (see, for instance, Ref.~\cite{ole:16}). The existence of the above sampling formula relies on the existence of dual frames for the Hilbert space product $\ell^2_{_N}(G)$ having the form $\{T_{g}\, \mathbf{b}_{m}\}_{g\in G;\, m=1,2,\dots,M}$, where $\mathbf{b}_m \in \ell^2_{_N}(G)$ and $T_{g}\mathbf{b}_m=\mathbf{b}_m(\cdot-g)$ denotes the translation operator in $\ell^2_{_N}(G)$. This can be reformulated as follows: given an analysis convolution system $\mathcal{A}: \ell^2_{_N}(G) \rightarrow \ell^2_{_M}(G)$ associated to data sampling, there exists another synthesis convolution system $\mathcal{B}: \ell^2_{_M}(G) \rightarrow \ell^2_{_N}(G)$ such that $\mathcal{B} \,\mathcal{A}=\mathcal{I}_{\ell^2_{_N}(G)}$. In other words, working in the Fourier domain $L^2(\widehat{G})$, we exploit the relationship between bounded convolution systems and frame theory in the product Hilbert  space $\ell^2_{_N}(G)$. All needed results on this relationship are included in Section \ref{section3}.

\medskip

Finally, it is worth to mention that most of the well known sampling results can be considered as particular examples of this approach; see Sections \ref{subsection4-2}--\ref{subsection4-4} for the details. A comparison with some previous similar sampling results is presented in Section \ref{subsection4-5}, where some affinities and differences are exhibited.

\section{Data samples as a filtering process}
\label{section2}
Let $\mathcal{H}$ be a separable Hilbert space, and let $G \ni g\mapsto U(g)\in \mathcal{U}(\mathcal{H})$ be a unitary representation of a countable discrete abelian group $(G,+)$ on  $\mathcal{H}$, i.e., it satisfies $U(g+g')=U(g)U(g')$, $U(-g)=U^{-1}(g)=U^*(g)$ for $g, g'\in G$. Given a set $\Phi=\{\varphi_{1},\varphi_{2,}\ldots,\varphi_{N}\}$ of generators in $\mathcal{H}$ we consider the subspace  $\mathcal{H}$ defined as
$\mathcal{V}_{\Phi}:=\overline{\text{span}}_{\mathcal{H}}\{U(g)\varphi_{n}\}_{g\in G;\,n=1, 2,\dots,N}$. Assuming that $\{U(g)\varphi_{n}\}_{g\in G;\,n=1, 2,\dots,N}$ is a Riesz sequence in $\mathcal{H}$, i.e., a Riesz basis for $\mathcal{V}_\Phi$, this subspace can be expressed as
\[
\mathcal{V}_{\Phi}=\Big\{ \sum_{n=1}^N\sum_{g\in G} x_{n}(g) U(g)\varphi_{n} \, :\,  x_{n}\in \ell^2(G),\,\, n=1, 2, \dots, N\Big\}\,.
\]
Let us motivate the sampling approach followed in this work by means of a couple of examples:

\noindent $\bullet$ Given $M$ elements $\psi_{m}\in \mathcal{H}$, $m=1, 2,\dots,M$, which do not belong necessarily to $\mathcal{V}_{\Phi}$, for any $f\in \mathcal{V}_{\Phi}$ we define for $m=1, 2, \dots, M$ its (generalized) average samples as
\begin{equation}
\label{samples1}
\mathcal{L}_{m}f(g):=\big\langle f, U(g)\psi_{m} \big\rangle_{\mathcal{H}},\quad g\in G\,.
\end{equation}
These samples can be expressed as the output of a convolution system. Indeed, for any $\displaystyle{f= \sum_{n=1}^N\sum_{g\in G} x_{n}(g) U(g)\varphi_{n}}$ in 
$\mathcal{V}_\Phi$, for each $m=1,2, \dots,M$ one immediately gets
\begin{equation}
\label{sampconvol}
\mathcal{L}_{m}f(g)=\sum_{n=1}^N (a_{m,n}\ast x_{n})(g),\,\,\, g\in G\,,
\end{equation}
with $a_{m,n}(g)=\langle \varphi_{n}, U(g)\psi_{m} \rangle_{\mathcal{H}}$\,, $g\in G$. Notice that each $a_{m,n}$ belongs to $\ell^2(G)$ since the sequence 
$\{U(g)\varphi_{n}\}_{g\in G;\,n=1, 2,\dots,N}$ is, in particular, a Bessel sequence in $\mathcal{H}$.

\medskip

\noindent $\bullet$ Suppose now that $\mathcal{H}=L^2(\mathbb{R}^d)$  and consider $M$ fixed points $t_m\in \mathbb{R}^d$, \, $ m=1,2, \dots, M$.  For each $f\in \mathcal{V}_{\Phi}$ we define formally its samples, for $m=1,2, \dots,M$, as
\begin{equation}
\label{samples2}
\mathcal{L}_{m}f(g):=\big[U(-g)f \big](t_m),\quad g\in G\,.
\end{equation}
It is straightforward to check that, for $\displaystyle{f= \sum_{n=1}^N\sum_{g\in G} x_{n}(g) U(g)\varphi_n}$ in $\mathcal{V}_{\Phi}$,  expression \ref{sampconvol} holds for  $a_{m,n}(g)=\big[U(-g)\varphi_n \big](t_m)$,\,\, $g\in G$. Under mild hypotheses (see Section \ref{subsection4-3}) one can obtain that $\mathcal{V}_\Phi$ is a reproducing kernel Hilbert space of continuous functions in $L^2(\mathbb{R}^d)$ where samples \eqref{samples2} are well defined with corresponding $a_{m,n}\in \ell^2(G)$, and yielding pointwise sampling in $\mathcal{V}_\Phi$.

\medskip

The above two situations englobe most of the regular (average or pointwise) sampling appearing in mathematical or engineering literature as we will see in Section \ref{section4}.

\bigskip

Consequently, one can think of a sampling process in  subspace $\mathcal{V}_\Phi$ as $M$ expressions like \eqref{sampconvol}, i.e., a convolution system $\mathcal{A}$ defined in the product Hilbert space $\ell^2_{_N}(G):=\ell^2(G)\times \dots \times \ell^2(G)$ ($N$ times) by means of a matrix 
$A=[a_{m,n}] \in \mathcal{M}_{_{M\times N}}(\ell^2(G))$, i.e., a $M\times N$ matrix  with entries in $\ell^2(G)$, as
\[
\mathcal{A}(\mathbf{x})= A  \ast \mathbf{x}\,,\quad \mathbf{x}=(x_1, x_2, \dots, x_N)^\top \in \ell^2_{_N}(G)\,,
\]
where $A  \ast \mathbf{x}$ denotes the (matrix) convolution
\[
(A  \ast \mathbf{x})(g)=\sum_{g'\in G} A(g-g')\, \mathbf{x}(g'),\quad g\in G\,.
\]
Note that the $m$-th entry of \,$A  \ast \mathbf{x}$\, is\, $\sum_{n=1}^N (a_{m,n}\ast x_{n})$, where $x_{n}$ denotes the $n$-th entry of $\mathbf{x} \in \ell^2_{_N}(G)$.

\medskip

The main aim in this paper is to recover, in a stable way, any $\mathbf{x} \in \ell^2_{_N}(G)$, or equivalently, the corresponding $\displaystyle{f= \sum_{n=1}^N\sum_{g\in G} x_{n}(g) U(g)\varphi_n}\in \mathcal{V}_\Phi$, from the vector data 
\[
\boldsymbol{\mathcal{L}}f(g):=\big(\mathcal{L}_1f(g), \mathcal{L}_2 f(g), \dots, \mathcal{L}_M f(g)\big)^\top=\big[\mathcal{A}(\mathbf{x})\big](g)\,, \quad g\in G\,,
\]
i.e., from the output $\mathcal{A}(\mathbf{x})$ of the convolution system $\mathcal{A}$ with associated matrix $A$, in case the vector sampling 
$\boldsymbol{\mathcal{L}}f \in \ell^2_{_M}(G)$.
\subsection{A brief on harmonic analysis on discrete abelian groups}
Let $(G, +)$ be a countable discrete abelian group and let $\mathbb{T}=\{z\in \mathbb{C}: |z|=1\}$ be the unidimensional torus. We say that $\xi:G\mapsto \mathbb{T}$ is a character of $G$ if $\xi(g+g')=\xi(g)\xi(g')$  for all $g,g'\in G$. We denote $\xi(g)=\langle g,\xi \rangle$.  By defining $(\xi+\xi')(g)=\xi(g)\xi'(g)$,  the set of characters $\widehat{G}$ is a group, called the dual group of $G$; since $G$ is discrete, the group 
$\widehat{G}$ is compact \cite[Prop. 4.4]{folland:95}. In particular, it is known that $\widehat{\mathbb{Z}}\cong \mathbb{T}$, with $\langle n,z \rangle = z^n$, and
$\widehat{\mathbb{Z}}_s\cong \mathbb{Z}_s:=\mathbb{Z}/s\mathbb{Z}$, with $\langle n,m \rangle = W_s^{nm}$, where $W_s=e^{2\pi i/s}$.

There exists a unique  measure, the Haar measure $\mu$ on  $\widehat{G}$ satisfying  $\mu(\xi+ E)=\mu(E)$,  for every Borel set $E\subset \widehat{G}$, and 
$\mu(\widehat{G})=1$. We denote 
$\int_{\widehat{G}} X(\xi) d\xi=\int_{\widehat{G}} X(\xi) d\mu(\xi)$. \\
If $G=\mathbb{Z}$, 
\[
\int_{\widehat{G}} X(\xi) d\xi=\int_{\mathbb{T}} X(z) dz= \frac{1}{2\pi}\int_0^{2\pi} X(e^{iw}) dw\,,
\]
and if $G=\mathbb{Z}_s$, 
\[
\int_{\widehat{G}} X(\xi) d\xi=\int_{\mathbb{Z}_s} X(n) dn= \frac{1}{s}\sum_{n\in \mathbb{Z}_s} X(n)\,.
\]
If $G_1, G_2,\ldots ,G_d$ are abelian discrete groups then the dual group of the product group is 
$\big(G_1 \times G_2 \times \ldots \times G_d\big)^{\wedge}\cong \widehat{G}_1 \times \widehat{G}_2 \times \ldots \times \widehat{G}_d$ with
\[
\big\langle\, (g_1,g_2,\ldots,g_d)\, ,\, (\xi_1,\xi_2\ldots,\xi_d)\, \big\rangle = \langle g_1,\xi_1\rangle \langle g_2,\xi_2\rangle\cdots \langle g_d,\xi_d\rangle\,.
\]
For $x\in \ell^1(G)$  its {\em Fourier transform} is defined as
\[
X(\xi)=\widehat{x}(\xi):=\sum_{g\in G}x(g) \overline{\langle g,\xi \rangle}=\sum_{g\in G}x(g) \langle -g,\xi \rangle\,,\quad \xi\in \widehat{G}\,.
\]
The Plancherel theorem extends uniquely the Fourier transform on $\ell^1(G)\cap \ell^2(G)$ to a unitary isomorphism from $\ell^2(G)$ to $L^2(\widehat{G})$. For the details see, for instance, Ref.~\cite{folland:95}.

\section{Convolution systems on discrete abelian groups}
\label{section3}
This section is devoted to collect some known results on discrete convolution systems, and to prove the new ones needed in the sequel. We will consider  bounded operators $\mathcal{A}:\ell^2_{_N}(G)\to \ell^2_{_M}(G)$ expressed as $\mathcal{A}(\mathbf{x})=A  \ast \mathbf{x}$ for each $\mathbf{x} \in \ell^2_{_N}(G)$.
For a fixed $g\in G$ we denote, as usually, the translation by $g$ of any $\mathbf{x} \in \ell^2_{_N}(G)$ as $T_{g}\mathbf{x}(h)=\mathbf{x}(h-g)$, $h\in G$. The first two results can be found in 
\cite[Thms. 2-3]{gerardo:19}

 \begin{prop}
 \label{bessel}  
 Given $A\in  \mathcal{M}_{_{M\times N}}\big(\ell^2(G)\big)$, the operator $\mathcal{A}: \mathbf{x} \mapsto A  \ast \mathbf{x}$ is a well defined bounded operator from 
 $\ell^2_{_N}(G)$ into $\ell^2_{_M}(G)$ if and only if  $\widehat{A}\in \mathcal{M}_{_{M\times N}}\big(L^\infty(\widehat{G})\big)$, where 
 \[
 \widehat{A}(\xi):=\big[\widehat{a}_{m,n}(\xi)\big]\,, \quad \text{a.e. $\xi \in \widehat{G}$}\,. 
 \] 
 \end{prop}
  
 \begin{prop}
 \label{convolution}  
 For a linear operator $\mathcal{A}:\ell^2_{_N}(G)\to \ell^2_{_M}(G)$ the following conditions are equivalent:
 \begin{itemize}
 \item[(a)] $\mathcal{A}$  is a bounded operator that conmutes with translations, i.e., $\mathcal{A}T_{g}=T_{g}\mathcal{A}$, for all $g\in G$.
  \item[(b)] There exists a matrix $A\in  \mathcal{M}_{_{M\times N}}\big(\ell^2(G)\big)$ that satisfies $\widehat{A} \in \mathcal{M}_{_{M\times N}}\big(L^\infty(\widehat{G})\big)$ and such that $\mathcal{A}(\mathbf{x})=A\ast \mathbf{x}$ for each $\mathbf{x}\in \ell^2_{_N}(G)$.
  \item[(c)]  There exists a matrix $\Lambda \in \mathcal{M}_{_{M\times N}}\big(L^\infty(\widehat{G})\big)$ such that $\widehat{\mathcal{A}(\mathbf{x})}=\Lambda \cdot \widehat{\mathbf{x}}$ for each $\mathbf{x}\in \ell^2_{_N}(G)$.
  \end{itemize}
  The matrices $A$ and $\Lambda$ satisfying (b) and (c) are unique and satisfy $\Lambda=\widehat{A}$.
  \end{prop}
  
\medskip

Under equivalent conditions in Prop.~\ref{convolution}, we say that $\mathcal{A}$ is a {\em bounded convolution operator}, and the unique matrix $\widehat{A}$ which satisfies $\widehat{\mathcal{A}(\mathbf{x})}=\widehat{A} \cdot \widehat{\mathbf{x}}$ for each 
$\mathbf{x}\in \ell^2_{_N}(G)$, is called the {\em transfer matrix} of the operator $\mathcal{A}$.

\medskip

\begin{prop}
\label{sobre} 
Let $\mathcal{A}: \ell^2_{_N}(G)\rightarrow \ell^2_{_M}(G)$ be a bounded convolution operator with transfer matrix $\widehat{A}$. Then: 

\begin{enumerate}[(a)]
\item The adjoint operator $\mathcal{A}^*$ is a bounded convolution operator with transfer matrix $\widehat{A}^*$, the adjoint matrix (transpose conjugate) of 
$\widehat{A}$, i.e., 
$\widehat{A}^*(\xi)=[\widehat{A}(\xi)]^*$, a.e. $\xi \in \widehat{G}$ (in the sequel $\widehat{A}(\xi)^*$). 

\item If $\mathcal{B}: \ell^2_{_M}(G)\rightarrow \ell^2_{_K}(G)$ is other bounded convolution operator with transfer matrix $\widehat{B}$ then the composition 
$\mathcal{B}\mathcal{A}: \ell^2_{_N}(G)\rightarrow \ell^2_{_K}(G)$ is a bounded convolution operator with transfer matrix $\widehat{B} \cdot \widehat{A}$.

\item $\displaystyle{\|\mathcal{A}\|=\esup_{\xi\in \widehat{G}} \| \widehat{A}(\xi)\|_{2}}$, where $\|\cdot\|_2$ denotes the spectral norm of the matrix.

\item $\mathcal{A}$ is injective with a closed range if and only if  $\einf_{\xi\in \widehat{G}} \det[\widehat{A}(\xi)^*\widehat{A}(\xi)]>0$.

\item $\mathcal{A}$ is onto if and only if $\einf_{\xi\in \widehat{G}} \det[\widehat{A}(\xi)\widehat{A}(\xi)^*]>0$.

\item $\mathcal{A}$ is an isomorphism if and only if  $M=N$ and $\einf_{\xi\in \widehat{G}} |\det\widehat{A}(\xi)|>0$. In this case,  
$\mathcal{A}^{-1}$ is a bounded convolution operator with transfer matrix $(\widehat{A}\,)^{-1}$ and
\[
\|\mathcal{A}^{-1}\|=\big(\einf_{\xi\in \widehat{G}} \lambda_{\min} [\widehat{A}(\xi)^*\widehat{A}(\xi)]\big)^{-1/2}.
\]
 \end{enumerate}
\end{prop}
\begin{proof}
\begin{enumerate}[(a)]
\item Using Prop.~\ref{convolution},  for each $\mathbf{x}\in \ell^2_{N}(G)$ and $\mathbf{y}\in \ell^2_{M}(G)$ we have 
\[
\langle\widehat{\mathbf{x}},\widehat{\mathcal{A}^*\mathbf{y}}\rangle_{L^2_{N}(\widehat{G})}=
\langle \mathbf{x},\mathcal{A}^*\mathbf{y}\rangle_{\ell_{N}^2(G)}=
\langle \mathcal{A}\mathbf{x},\mathbf{y}\rangle_{\ell_{M}^2(G)}=
\langle\widehat{A}\cdot\widehat{\mathbf{x}},\widehat{\mathbf{y}}\rangle_{L^2_{M}(\widehat{G})}
=
\langle\widehat{\mathbf{x}},\widehat{A}^*\cdot\widehat{\mathbf{y}}\rangle_{L^2_{N}(\widehat{G})}\,.
\]
Hence $\widehat{\mathcal{A}^*\mathbf{y}}=\widehat{A}^*\cdot\widehat{\mathbf{y}}$ for all
$\mathbf{y}\in \ell^2_{N}(G)$, and the result follows from Prop.~\ref{convolution}.

\item For each $\mathbf{x}\in \ell_{N}^2(G)$ we have that $\widehat{\mathcal{B}\mathcal{A}(\mathbf{x})}=\widehat{B}\cdot \widehat{\mathcal{A}(\mathbf{x})}=\widehat{B}\cdot \widehat{A}\cdot\widehat{\mathbf{x}}$. Since the entries of $\widehat{A}$ and  $\widehat{B}$ belong to $L^\infty(\widehat{G})$, we get that 
$\widehat{B}\cdot \widehat{A}\in  \mathcal{M}_{_{K\times N}}(L^\infty(\widehat{G}))$, and the result follows from Prop.~\ref{convolution}.
\item The result is proved in \cite[Cor.~6]{gerardo:19} for the case $N=M$. Hence, we obtain
\[
\|\mathcal{A}\|^2=\|\mathcal{A^*A}\|=\esup_{\xi\in \widehat{G}} \| \widehat{A}(\xi)^*\widehat{A}(\xi)\|_{2}=
\esup_{\xi\in \widehat{G}} \| \widehat{A}(\xi)\|^2_{2}\,.
\] 
\item A bounded operator $\mathcal{A}$ between Hilbert spaces is injective with a closed range if and only if the operator $\mathcal{A}^*\mathcal{A}$ is invertible. 
By using (a) and (b), we have that  $\mathcal{A}^*\mathcal{A}$ is a  bounded convolution operator with transfer matrix $ \widehat{A}(\xi)^*\widehat{A}(\xi)$, and the result follows from \cite[Thm.~7]{gerardo:19}.
\item A bounded operator $\mathcal{A}$ is onto if and only if its adjoint operator $\mathcal{A}^*$  is injective with a closed range; from (a), the transfer matrix of $\mathcal{A}^*$ is 
$\widehat{A}^*$. Thus, the result follows from (d). 
\item This characterization is a consequence of (d) and (e). From \cite[Thm.~7]{gerardo:19}, the inverse operator $\mathcal{A}^{-1}$ is a  bounded convolution operator  with transfer matrix 
$\widehat{A}(\xi)^{-1}$ and norm $\|\mathcal{A}^{-1}\|=\big(\einf_{\xi\in \widehat{G}} \lambda_{\min} [\widehat{A}(\xi)^*\widehat{A}(\xi)]\big)^{-1/2}$ 
\end{enumerate}
\end{proof}

\medskip

Note that from (a) the matrix associated with the adjoint operator $\mathcal{A}^*$ is not the adjoint matrix of $A$, but the one defined by means of the involution 
 \begin{equation}
 \label{involution}
A^*=\big[a^*_{m,n}\big]^\top \in  \mathcal{M}_{_{N\times M}}\big(\ell^2(G)\big)\quad\text{where}\quad a^*_{m,n}(g):=\overline{a_{m,n}(-g)}, \quad g\in G.
\end{equation}
Indeed, since $\widehat{a^*_{m,n}}(\xi)=\overline{\widehat{a}_{m,n}(\xi)}$, we have $\widehat{A^*}(\xi)=\widehat{A}(\xi)^*=\widehat{A}^*(\xi)$, a.e. $\xi \in \widehat{G}$.

\subsection{Dual frames  in $\ell^2_{_N}(G)$ having the form $\{T_{g}\, \mathbf{b}_{m}\}_{g\in G;\, m=1,2,\dots,M}$}
\label{section3-1}
Given a matrix $B\in \mathcal{M}_{_{N\times M}}(\ell^2(G))$,  the associated convolution operator $\mathcal{B} : \ell^2_{_M}(G) \rightarrow \ell^2_{_N}(G)$  can be written in terms of its M columns $\mathbf{b}_{1}, \mathbf{b}_{2}\dots,\mathbf{b}_{M}$, as
\begin{equation}
\label{sintesis}
\mathcal{B}(\mathbf{x})= B \ast \mathbf{x}=
\sum_{m=1}^M\sum_{g\in G} x_m(g) T_{g}\mathbf{b}_{m}\,, \quad \mathbf{x} \in \ell^2_{_M}(G)\,,
\end{equation}
where $T_{g} \mathbf{b}_m = \mathbf{b}_m(\cdot-g)$ denotes the translation operator for each $m=1,2, \dots, M$.
In other words, operator $\mathcal{B}$ is the {\em synthesis operator} of the sequence $\big\{T_{g} \mathbf{b}_{m}\big\}_{g\in G;\, m=1,2,\dots,M}$ in $\ell^2_{_N}(G)$. 

\medskip

\noindent Thus Props.~\ref{bessel}, \ref{convolution} and \ref{sobre}  can be translated (interchanging $M$ by $N$)  to the associated sequence $\big\{T_{g} \mathbf{b}_{m}\big\}_{g\in G;\, m=1,2,\dots,M}$. For instance, since a sequence in a Hilbert space is a Bessel sequence if and only if its synthesis operator is bounded and, in this case, its optimal Bessel bound is the square of the synthesis operator norm \cite{ole:16}, from Props.~\ref{bessel} and \ref{sobre}  we get

\begin{prop}
\label{be}
The sequence $\big\{T_g \mathbf{b}_{m}\big\}_{g\in G;\, m=1,2,\dots,M}$  is a Bessel sequence for $\ell^2_{_N}(G)$ if and only if the transfer matrix $\widehat{B}$ belongs to  $\mathcal{M}_{_{N\times M}}\big(L^\infty(\widehat{G})\big)$. In this case the optimal Bessel bound is
$\beta_{_B}=\esup_{\xi\in \widehat{G}} \| \widehat{B}(\xi)\|^2_{2}$.   
\end{prop}

Let $\mathbf{a}^*_{m}$ denote the $m$-th column of the matrix $A^*$, the associated matrix  of  $\mathcal{A}^*$, given in \eqref{involution}. The convolution operator $\mathcal{A}:\ell_{N}^2(G)\to \ell_{M}^2(G)$ can also be written as
\begin{equation}
\label{analisis}
[\mathcal{A}(\mathbf{x})]_{m}(g)=[A\ast \mathbf{x}]_{m}(g)=\big\langle \mathbf{x}, T_{g}\mathbf{a}^*_{m} \big\rangle_{\ell^2_{_N}(G)}\,.
\end{equation}
In other words, operator $\mathcal{A}$ is the {\em analysis operator} for the sequence $\big\{T_{g} \mathbf{a}^*_{m}\big\}_{g\in G;\, m=1,2,\dots,M}$ in $\ell^2_{N}(G)$. 

\begin{prop}
\label{frame} 
Assume that $\widehat{A}\in \mathcal{M}_{_{M\times N}}(L^\infty(\widehat{G}))$. Then:
\begin{itemize}
\item[(a)] The sequence $\big\{T_g \mathbf{a}^*_{m}\big\}_{g\in G;\, m=1,2,\dots,M}$
 is a frame for $\ell^2_{_N}(G)$ if and only if 
\[\einf_{\xi\in \widehat{G}} \det \big[\widehat{A}(\xi)^*\widehat{A}(\xi)\big]>0.\] 
In this case, the optimal frame bounds are
\[
\alpha_{_{A}}=\einf_{\xi\in \widehat{G}} \lambda_{\min} [\widehat{A}(\xi)^*\widehat{A}(\xi)]\quad \text{and} \quad
\beta_{_{A}}=\esup_{\xi\in \widehat{G}} \lambda_{\max} [\widehat{A}(\xi)^*\widehat{A}(\xi)]\,.
\]
\item[(b)] The sequence $\big\{T_g \mathbf{a}^*_{m}\big\}_{g\in G;\, m=1,2,\dots,M}  $ is a Riesz basis for $\ell^2_{_N}(G)$ if and only if $N=M$ and
$\displaystyle\einf_{\xi\in \widehat{G}} \big|\det [\widehat{A}(\xi)]\big|>0$.
\end{itemize}
\end{prop}
\begin{proof}
$(a)$ Since $\widehat{A}\in \mathcal{M}_{_{M\times N}}(L^\infty(\widehat{G}))$, from Prop.~\ref{be}  the sequence $\big\{T_g \mathbf{a}^*_{m}\big\}_{g\in G;\, m=1,2,\dots,M}$, whose corresponding transfer matrix is $\widehat{A}^*$, is a Bessel sequence of $\ell^2_{_N}(G)$. Since a Bessel sequence is a frame  if and only if its analysis operator $\mathcal{A}$  is injective with a closed range (see, for instance, Ref. \cite{ole:16}), the result is a consequence of Prop.~\ref{sobre}(d). Since the optimal upper frame bound 
$\beta_{_A}$ is the squared norm of the analysis operator $\mathcal{A}$, and the optimal  lower frame bound $\alpha_{_{A}}$  is the reciprocal of the norm of  the inverse of the frame operator $\mathcal{A}^*\mathcal{A}$ (see, for instance, Ref.~\cite{ole:16}), from Prop.~\ref{sobre} we get 
\[
\begin{split}
\beta_{_{A}}&=\|\mathcal{A}\|^2=\esup_{\xi\in \widehat{G}} \|\widehat{A}(\xi)\|^2_{2}=\esup_{\xi\in \widehat{G}} \lambda_{\max} [\widehat{A}(\xi)^*\widehat{A}(\xi)],
\\
\alpha_{_{A}}&=\|(\mathcal{A}^*\mathcal{A})^{-1}\|^{-1}=
 \big(\einf_{\xi\in \widehat{G}} \lambda_{\text{min}} [\widehat{A}(\xi)^*\widehat{A}(\xi)\widehat{A}(\xi)^*\widehat{A}(\xi)]\big)^{1/2}
 =\einf_{\xi\in \widehat{G}} \lambda_{\text{min}} [\widehat{A}(\xi)^*\widehat{A}(\xi)]\,.
 \end{split}
 \]
$(b)$ The Bessel sequence $\big\{T_g \mathbf{a}^*_{m}\big\}_{g\in G;\, m=1,2,\dots,M}$ is a Riesz basis for $\ell^2_{_N}(G)$ if and only if its synthesis operator $\mathcal{A}^*$ is  an isomorphism (see, for instance, Ref.~\cite{ole:16}). Hence, the result is a consequence of Prop.~\ref{sobre}(f).
\end{proof}

\begin{prop}
\label{dualframe} 
Assume that $\widehat{A}\in \mathcal{M}_{_{M\times N}}(L^\infty(\widehat{G}))$ and $\widehat{B}\in \mathcal{M}_{_{N\times M}}(L^\infty(\widehat{G}))$. Then the sequences
 $\big\{T_g \mathbf{a}^*_{n}\big\}_{g\in G;\, n=1,2,\dots,M}$ and $\big\{T_g \mathbf{b}_{n}\big\}_{g\in G;\, n=1,2,\dots,M}$ form a pair of dual frames for $\ell^2_{_N}(G)$ if and only if \[
\widehat{B}(\xi)\,\widehat{A}(\xi)=I_{_N}\,,\quad \text{a.e. $\xi \in \widehat{G}$}\,.
 \]
\end{prop}
\begin{proof}
Having in mind that the analysis operator of $\big\{T_g \mathbf{a}^*_{n}\big\}_{g\in G;\, n=1,2,\dots,M}$  is $\mathcal{A}$ and that the synthesis operator of $\big\{T_g \mathbf{b}_{n}\big\}_{g\in G;\, n=1,2,\dots,M}$ is $\mathcal{B}$, we obtain that these two Bessel sequences form a pair of dual frames if and only if  $\mathcal{B}\mathcal{A}=\mathcal{I}_{\ell^2_{N}(G)}$ \cite[Lemma 6.3.2]{ole:16} or, equivalently, $\widehat{B}(\xi)\,\widehat{A}(\xi)=I_{_N}$, a.e. $\xi \in \widehat{G}$. (see Prop.~\ref{sobre}(b)).
\end{proof}
\section{The resulting sampling theory}
\label{section4}
In this section we propose a regular sampling theory for a multiply generated $U$-invariant subspace $\mathcal{V}_\Phi$ in a separable Hilbert space $\mathcal{H}$. This theory includes most of classical well known regular sampling results for shift-invariant subspaces of $L^2(\mathbb{R}^d)$. Besides, we obtain new sampling results; for instance, those associated with crystallographic groups.
\subsection{Sampling in a $U$-invariant subspace with multiple generators}
\label{subsection4-1}
Suppose that $g \mapsto U(g)$ is a unitary representation of the countable discrete abelian group $G$ on a separable Hilbert space $\mathcal{H}$, and assume that for a fixed set of generators $\Phi=\{\varphi_{1},\varphi_{2,}\ldots,\varphi_{N}\}$ in $\mathcal{H}$ the sequence $\big\{U(g)\varphi_n\big\}_{g\in G;\, n=1,2,\dots,N}$ is a Riesz sequence for $\mathcal{H}$. 
For necessary and sufficient conditions see Ref.~\cite{gerardo:19}; see also Refs.~\cite{aldroubi:96,barbieri:15, cabrelli:10,lee:93,jia:91}. 
Thus, we consider the $U$-invariant subspace in $\mathcal{H}$
\begin{equation}
\label{invariantsubspace}
\mathcal{V}_\Phi=\Big\{\sum_{n=1}^N\sum_{g\in G} x_{n}(g) U(g)\varphi_{n} \, :\,  x_{n}\in \ell^2(G),\,\, n=1, 2, \dots, N\Big\}\,.
\end{equation}
For a given matrix $A=[a_{m,n}] \in \mathcal{M}_{_{M\times N}}\big(\ell^2(G)\big)$, we consider the vector samples 
$\boldsymbol{\mathcal{L}}f$ of any
$f=\sum_{n=1}^N\sum_{g\in G} x_{n}(g) U(g)\varphi_{n}\in \mathcal{V}_\Phi$ defined by 
\begin{equation}\
\label{genalizedsamples}
\boldsymbol{\mathcal{L}}f(g):=\big(\mathcal{L}_1f(g), \mathcal{L}_2 f(g), \dots, \mathcal{L}_M f(g)\big)^\top=(A \ast \mathbf{x})(g)=\big[\mathcal{A}(\mathbf{x})\big](g)\,, \quad g\in G\,.
\end{equation}

Assume $\widehat{A}\in \mathcal{M}_{_{M\times N}}(L^\infty(\widehat{G}))$ and $\einf_{\xi\in \widehat{G}} \det \big[\widehat{A}(\xi)^*\widehat{A}(\xi)\big]>0$. Since $[\mathcal{A}(\mathbf{x})]_{m}(g)=\big\langle \mathbf{x}, T_{g}\mathbf{a}^*_{m} \big\rangle_{\ell^2_{_N}(G)}$, the optimal frame bounds given in Prop.~\ref{frame} provide relevant information about the stability of the recovering. Namely,
\[
\alpha_{_{A}}\, \|\mathbf{x}\|^2 \le  \sum_{m=1}^M \sum_{g\in G}|\mathcal{L}_{m}f(g)|^2 \le \beta_{_{A}} \|\mathbf{x}\|^2\,, \quad \mathbf{x} \in \ell^2_{_N}(G)\,,
\]
where
$
\alpha_{_{A}}=\einf_{\xi\in \widehat{G}} \lambda_{\min} [\widehat{A}(\xi)^*\widehat{A}(\xi)]$ and 
$\beta_{_{A}}=\esup_{\xi\in \widehat{G}} \lambda_{\max} [\widehat{A}(\xi)^*\widehat{A}(\xi)]$. Moreover, denoting by $\alpha_\Phi$ and $\beta_\Phi$ the Riesz bounds for $\big\{U(g)\varphi_n\big\}_{g\in G;\, n=1,2,\dots,N}$ (see \cite[Thm. 9]{gerardo:19}) we have
\[
\alpha_\Phi\alpha_{_{A}}\, \|f\|^2 \le \sum_{m=1}^M \sum_{g\in G}|\mathcal{L}_{m}f(g)|^2 \le \beta_\Phi \beta_{_{A}} \|f\|^2\,, \quad f\in \mathcal{V}_\Phi\,.
\]
Now, for the recovery of any $f\in \mathcal{V}_\Phi$ from its generalized samples \eqref{genalizedsamples}, the idea is to find a $N\times M$ matrix $\widehat{B}\in \mathcal{M}_{_{N\times M}}(L^\infty(\widehat{G}))$ such that $\widehat{B}(\xi)\,\widehat{A}(\xi)=I_{_N}$, a.e. $\xi \in \widehat{G}$. In other words,  the corresponding convolution operator $\mathcal{B}(\mathbf{x})=B\ast \mathbf{x}$ should satisfy $\mathbf{x}=\mathcal{B}\mathcal{A}(\mathbf{x})=\mathcal{B}(\boldsymbol{\mathcal{L}}f)$, that is
\[
\mathbf{x}=B\ast \boldsymbol{\mathcal{L}}f\quad \text{and} \quad
f= \sum_{n=1}^N\sum_{g\in G} x_{n}(g) U(g)\varphi_{n}\,,\quad f\in \mathcal{V}_\Phi\,.
\]
Moreover, an explicit structured sampling formula can be obtained. Namely,  the recovering of $\mathbf{x}=B\ast \boldsymbol{\mathcal{L}}f$ from the samples $\boldsymbol{\mathcal{L}}f$ can be written as an expansion in terms of a pair of dual frames (see  Eqs.~\eqref{sintesis}--\eqref{analisis} and Prop.~\ref{dualframe})
\begin{equation}
\label{expansion1}
\mathbf{x}=\sum_{m=1}^M \sum_{g\in G} \big\langle \mathbf{x}, T_g \mathbf{a}^*_{m} \big \rangle_{\ell^2_{_N}(G)}\, T_g \mathbf{b}_{m}=\sum_{m=1}^M\sum_{g\in G} \mathcal{L}_{m}f(g)\, T_g \mathbf{b}_{m} \quad \text{in $\ell^2_{_N}(G)$}\,.
\end{equation} 
Besides, we consider the natural isomorphism $\mathcal{T}_{U,\Phi}: \ell^2_{_N}(G) \rightarrow \mathcal{V}_\Phi$ which maps the standard orthonormal basis 
$\{\boldsymbol{\delta}_{g,n}\}_{g\in G;\, n=1,2,\dots,N}$ for $\ell^2_{_N}(G)$ onto the Riesz basis $\big\{U(g)\varphi_n\big\}_{g\in G;\, n=1,2,\dots,N}$ for $\mathcal{V}_\Phi$. This isomorphism satisfies the {\em shifting property}:
\begin{equation}
\label{shifting}
\mathcal{T}_{U,\Phi}\big(T_g \mathbf{b} \big)=U(g)\big(\mathcal{T}_{U,\Phi}\mathbf{b} \big)\, \quad \text{for each $g\in G$ and $\mathbf{b}\in \ell^2_{_N}(G)$}\,.
\end{equation} 
Finally, for each $f=\mathcal{T}_{U,\Phi} \mathbf{x}\in \mathcal{V}_\Phi$, applying the isomorphism $\mathcal{T}_{U,\Phi}$ on \eqref{expansion1} and the shifting property \eqref{shifting} we obtain the sampling expansion  in $\mathcal{V}_\Phi$
\begin{equation}
\begin{split}
\label{sampling1}
f&=\sum_{m=1}^M\sum_{g\in G} \mathcal{L}_{m}f(g)\, \mathcal{T}_{U,\Phi}\big(T_g \mathbf{b}_m\big)=\sum_{m=1}^M\sum_{g\in G} \mathcal{L}_{m}f(g)\, U(g)\big(\mathcal{T}_{U,\Phi}\mathbf{b}_m \big)\\
&=\sum_{m=1}^M\sum_{g\in G} \mathcal{L}_{m}f(g)\, U(g)S_{m}\quad \text{in $\mathcal{H}$}\,,
\end{split}
\end{equation}
where the reconstruction elements are given by  $S_{m}=\mathcal{T}_{U,\Phi}\mathbf{b}_m \in \mathcal{V}_\Phi$, \, $m=1,2,\dots,M$, and the sequence $\{U(g)S_{m}\}_{g\in G;\, m=1,2,\dots,M}$ is a frame for $\mathcal{V}_\Phi$.
In fact, the following sampling theorem in the subspace $\mathcal{V}_\Phi$ holds:

\begin{teo}
\label{Usampling}
Let $A=[a_{m,n}]\in \mathcal{M}_{_{M\times N}}\big(\ell^2(G)\big)$ be the matrix defining the samples $\boldsymbol{\mathcal{L}}f(g)$, $g\in G$, for each $f\in \mathcal{V}_\Phi$ as in \eqref{genalizedsamples}, and assume that its transfer matrix $\widehat{A}$ has all its entries in $L^\infty(\widehat{G})$. Then, the following statements are equivalent:
\begin{enumerate}[(a)]
\item The constant $\displaystyle{\delta_{_A}:=\einf_{\xi \in \widehat{G}} \det [\widehat{A}(\xi)^*\widehat{A}(\xi)]>0}$.

\item There exist constants $0<\alpha \le\beta$ such that
\[
\alpha \|f\|^2 \le \sum_{m=1}^M \sum_{g\in G}|\mathcal{L}_{m}f(g)|^2 \le \beta \|f\|^2\,, \quad f\in \mathcal{V}_\Phi\,.
\]

\item There exists a matrix $\widehat{B}\in \mathcal{M}_{_{N\times M}}\big(L^\infty(\widehat{G})\big)$ such that $\widehat{B}(\xi)\,\widehat{A}(\xi)=I_{_N}$, a.e. $\xi \in \widehat{G}$.

\item There exists a matrix $B\in \mathcal{M}_{_{N\times M}}(\ell^2(G))$ with $\widehat{B}\in \mathcal{M}_{_{N\times M}}(L^\infty(\widehat{G}))$, such that
\[
\mathbf{x}=B\ast \boldsymbol{\mathcal{L}}f\quad \text{and} \quad f= \sum_{n=1}^N\sum_{g\in G} x_{n}(g) U(g)\varphi_{n}\,, \quad f\in \mathcal{V}_\Phi\,.
\]
In other words, there exists a bounded convolution system $\mathcal{B}: \ell^2_{_M}(G) \rightarrow \ell^2_{_N}(G)$ such that $\mathcal{B} \mathcal{A}=\mathcal{I}_{\ell^2_{_N}(G)}$.
\item There exist $M$ elements $S_m\in \mathcal{V}_\Phi$ such that the sequence $\big\{U(g)S_m\big\}_{g\in G;\,m=1,2,\dots, M}$ is a frame for $\mathcal{V}_\Phi$ and for each $f\in \mathcal{V}_\Phi$ the reconstruction formula
\begin{equation}
\label{sampling2}
f=\sum_{m=1}^M\sum_{g\in G} \mathcal{L}_m f(g)\,U(g)S_m \quad \text{ in $\mathcal{H}$}
\end{equation}
holds. 
\item There exists a frame $\big\{S_{g,m}\big\}_{g\in G;\,m=1,2,\dots, M}$ for $\mathcal{V}_\Phi$ such that for each $f\in \mathcal{V}_\Phi$ the expansion
\[
f=\sum_{m=1}^M\sum_{g\in G} \mathcal{L}_m f(g)\,S_{g,m} \quad \text{ in $\mathcal{H}$}
\]
holds.
\end{enumerate}
In this case, the reconstruction elements $\{S_m\}_{m=1, 2, \dots,M}$ in $\mathcal{V}_\Phi$ in formula \eqref{sampling2} are necessarily obtained from the columns 
$\mathbf{b}_1, \mathbf{b}_2,\dots,\mathbf{b}_{M}$ of a matrix  $B$ satisfying $(c)$, i.e., $S_m=\mathcal{T}_{U,\Phi}\mathbf{b}_m=\sum_{n=1}^N \sum_{g\in G} b_{n,m}\,U(g)\varphi_n$, $m=1, 2, \dots,M$. 
\end{teo}
\begin{proof} 
First we note that,  since $\beta_{_A}<\infty$, condition $\delta_{_A}>0$ is equivalent to condition $\alpha_{_A}>0$.
Now we prove that $(a)$ and $(b)$ are equivalent. Indeed, since $\mathcal{T}_{U,\Phi}$ is an isomorphism condition (b) is equivalent to the existence of $0<\alpha_{1}\le \beta_{1}$ such that
\[
\alpha_{1}\|\mathbf{x}\|^2 \le \sum_{m=1}^M \sum_{g\in G}|\mathcal{L}_{m}f(g)|^2 \le \beta_{1} \|\mathbf{x}\|^2\,,\quad \mathbf{x} \in \ell^2_{_N}(G)\,.
\] 
Since Eq.~\eqref{analisis}, this is equivalent to be the sequence $\big\{T_{g} \mathbf{a}^*_{m}\big\}_{g\in G;\, m=1,2,\dots,M}$ a frame for $\ell_{N}^2(G)$. Therefore, the result follows from Prop.~\ref{frame}.

Assume now that $(a)$ holds. Then, from Prop. \ref{sobre}, operator $\mathcal{A}^*\mathcal{A}$ is invertible, and  $\mathcal{B}:=(\mathcal{A}^*\mathcal{A})^{-1}\mathcal{A}^*$ is a bounded convolution operator satisfying $\mathcal{B}\mathcal{A}=\mathcal{I}_{\ell^2_{_N}(G)}$. From Prop.~\ref{sobre}, its transfer matrix satisfies the requirement in $(c)$. 

If $\widehat{B}$ satisfies $(c)$, the bounded convolution operator $\mathcal{B}$ whose transfer matrix is  $\widehat{B}$ satisfies  
$\mathcal{B}\mathcal{A}=\mathcal{I}_{\ell^2_{_N}(G)}$ from Prop.~\ref{sobre}, that is, condition $(d)$.

 We have proved that condition $(d)$ implies a sampling expansion as \eqref{sampling2}, where $S_m=\mathcal{T}_{U,\Phi}\mathbf{b}_m$, $m=1, 2, \dots,M$, and $\mathbf{b}_1,\ldots,\mathbf{b}_{M}$ are the columns of a matrix $B$ satisfying $(d)$. Besides,  the sequence $\big\{U(g)S_m\big\}_{g\in G;\,m=1,2,\dots, M}= \mathcal{T}_{U,\Phi} \big\{T_g \mathbf{b}_{m}\big\}_{g\in G;\, m=1,2,\dots,M}$ is a frame since  \eqref{expansion1} is a frame expansion in $\ell^2_{_N}(G)$ and $\mathcal{T}^{-1}_{U,\Phi}$ an isomorphism. This proves condition $(e)$ which trivially implies condition $(f)$. 

Finally, condition $(f)$ implies $(a)$. Applying  $\mathcal{T}^{-1}_{U,\Phi}$ to the formula in $(f)$ we obtain that $\big\{T_g \mathbf{a}^*_{m}\big\}_{g\in G;\, m=1,2,\dots,M}$ and 
$\{\mathcal{T}^{-1}_{U,\Phi} S_{g,m}Ê\}_{g\in G;\, m=1,2,\dots,M}$ form a pair of dual frames for $\ell^2_{_N}(G)$; in particular, by using Prop.~\ref{frame}(a) we obtain that 
$\delta_{_A}>0$.
\end{proof}

All the possible solutions of $\widehat{B}(\xi)\widehat{A}(\xi)=I_{_N}$ a.e. $\xi \in \widehat{G}$ with entries in $L^\infty(\widehat{G})$ are given in terms of the Moore-Penrose pseudo-inverse $\widehat{A}(\xi)^\dag=\big[\widehat{A}(\xi)^*\widehat{A}(\xi)\big]^{-1}\widehat{A}(\xi)^*$ by means of the $N\times M$ matrices $\widehat{B}(\xi):=\widehat{A}(\xi)^\dag+C(\xi)\big[I_M-\widehat{A}(\xi)\widehat{A}(\xi)^\dag\big]$,
where $C(\xi)$ denotes any $N\times M$ matrix with entries in $L^\infty(\widehat{G})$. Since $\mathbf{x}=B\ast \boldsymbol{\mathcal{L}}f$, 
from Prop.~\ref{bessel} we have that
\[
\|\mathbf{x}\|_{\ell^2_N(G)}^2\le C \sum_{m=1}^M \sum_{g\in G}|\mathcal{L}_{m}f(g)|^2\quad \text{where}\,\,\, C = \esup_{\xi\in \widehat{G}} \|\widehat{B}(\xi)\|^2_2\,.
\]
The best  possible bound we can get is $C=\alpha_{_{A}}^{-1}=\einf_{\xi\in \widehat{G}} \lambda_{\min} [\widehat{A}(\xi)^*\widehat{A}(\xi)]^{-1}$, which correspond to choosing   $\widehat{B}=\widehat{A}^\dag$ or, equivalently, $B=(A^*A)^{-1}A^*$, the pseudo-inverse of $A$ (see Ref.~\cite{ole:16}).

\medskip

Notice that in Thm.~\ref{Usampling} necessarily $M\geq N$ where $N$ is the number of generators in $\mathcal{V}_\Phi$. In case $M=N$, we have:
\begin{cor}
In case $M=N$, assume that the transfer matrix $\widehat{A}(\xi)$ has all entries in $L^\infty(\widehat{G})$. The following statements are equivalent:
\begin{enumerate}
\item The constant $\displaystyle\einf_{\xi\in \widehat{G}} \big|\det [\widehat{A}(\xi)]\big|>0$.
\item There exist $N$ unique elements $S_n$, $n=1,2,\dots, N$, in $\mathcal{V}_\Phi$ such that the associated sequence
$\big\{U(g)S_{n}\big\}_{g\in G;\,n=1,2,\dots, N}$ is a Riesz basis for $\mathcal{V}_\Phi$ and the sampling formula
\[
f=\sum_{n=1}^N\sum_{g\in G} \mathcal{L}_n f(g)\,U(g)S_{n}\ \quad \text{ in $\mathcal{H}$}
\]
holds for each $f\in \mathcal{V}_\Phi$.
\end{enumerate}
Moreover, the interpolation property $\mathcal{L}_n S_{n'}(g)=\delta_{n,n'}\delta_{g,0_G}$, where $g\in G$ and $n,n'=1,2,\dots,N$, holds.
\end{cor}
\begin{proof}
In this case, the square matrix $\widehat{A}(\xi)$ is invertible and the result comes out from Prop.~\ref{frame}(b). The uniqueness of the coefficients in a Riesz basis expansion gives the interpolation property.
\end{proof}
\subsubsection{A more general framework}
A slightly more general setting is motivated by condition $(f)$ in Thm.~\ref{Usampling}. Namely, let $F:=\{f_{g,n}\}_{g\in G;\, n=1,2,\dots,N}$ be a Riesz sequence in a separable Hilbert space 
$\mathcal{H}$, and let $\mathcal{V}_F:=\overline{\espan}_\mathcal{H} \{f_{g,n}\}_{g\in G;\, n=1,2,\dots,N}$ be its associated subspace, that is,
\[
\mathcal{V}_F=\Big\{\sum_{n=1}^N\sum_{g\in G} x_{n}(g)\, f_{g,n} \, :\,  x_{n}\in \ell^2(G),\,\, n=1, 2, \dots, N\Big\}\,.
\]
Given a  matrix $A=[a_{m,n}]\in \mathcal{M}_{_{M\times N}}\big(\ell^2(G)\big)$, for each $f= \sum_{n=1}^N\sum_{g\in G} x_{n}(g) \,f_{g,n}$ in $\mathcal{V}_F$ we define its data samples $\boldsymbol{\mathcal{L}} f$ by means of $A$ and $\mathbf{x}=(x_1, x_2, \dots, x_N)^\top \in \ell^2_{_N}(G)$ as
\[
\boldsymbol{\mathcal{L}} f(g):=\big( \mathcal{L}_1 f(g), \mathcal{L}_2 f(g), \dots,\mathcal{L}_M f(g)\big)^\top=(A\ast \mathbf{x})(g)\,,\quad g\in G\,.
\]
As before, the aim is the stable recovery of any $f\in \mathcal{V}_F$ from data $\boldsymbol{\mathcal{L}} f \in \ell^2_{_M}(G)$. Under the hypotheses on the matrix $A$ in Thm.\ref{Usampling} there exists a frame $\{S_{g,m}\}_{g\in G;\,m=1,2,\dots,M}$ for $\mathcal{V}_F$ such that for each $f\in \mathcal{V}_F$ the reconstruction formula
\[
f=\sum_{m=1}^M\sum_{g\in G} \mathcal{L}_m f(g)\,S_{g,m} \quad \text{ in $\mathcal{H}$}
\]
holds. Moreover, there exist $\mathbf{b}_m$ in $\ell^2_{_N}(G)$, $m=1,2,\dots,M$, such that $S_{g,m}=\mathcal{T}_F\big(T_g \mathbf{b}_m\big)$,\, $g\in G$ and $m=1,2,\dots,M$, where 
$\mathcal{T}_F: \ell^2_{_N}(G) \rightarrow \mathcal{V}_F$ stands for the natural isomorphism which maps the standard orthonormal basis $\{\boldsymbol{\delta}_{g,n}\}_{g\in G;\, n=1,2,\dots,N}$ for $\ell^2_{_N}(G)$ on the Riesz basis $\{f_{g,n}\}_{g\in G;\, n=1,2,\dots,N}$ for $\mathcal{V}_F$. Since the subspace $\mathcal{V}_F$ has not any a priori structure, the same occurs for the reconstruction functions $S_{g,m}$.
 
\subsection{Some regular sampling settings as particular examples}
\label{subsection4-2}
In this section, we illustrate the result in Thm.~\ref{Usampling} with some average sampling examples.

$\bullet$ Choose $\mathcal{H}:=L^2(\mathbb{R}^d)$, $G:=\mathbb{Z}^d$ and $\big(U(p)f\big)(t):=f(t-p)$, $t\in \mathbb{R}^d$ and $p\in \mathbb{Z}^d$. Under the hypotheses in Thm.~\ref{Usampling} for the average samples given by \eqref{samples1}, i.e., for the associated matrix $A=[a_{m,n}]$ where $a_{m,n}(p)=\langle \varphi_n, \psi_m(\cdot-p)\rangle_{L^2(\mathbb{R}^d)}$, we obtain oversampled {\em average sampling} in the classical shift-invariant subspace $V_\Phi^2$ of $L^2(\mathbb{R}^d)$ described as
\[
V_\Phi^2=\Big\{\sum_{n=1}^N\sum_{p\in \mathbb{Z}^d} x_{n}(p) \, \varphi_n(t-p)\, 
:\, x_n \in \ell^2(\mathbb{Z}^d),\, n=1,2,\dots,N \Big\}\,.
\]
Under mild hypotheses, the space $V_\Phi^2$ is a reproducing kernel Hilbert space (RKHS). For each $f\in V_\Phi^2$ a sampling expansion having the form
\[
f(t)=\sum_{m=1}^M\sum_{p\in \mathbb{Z}^d} \big\langle f, \psi_m(\cdot-p)\big\rangle_{L^2(\mathbb{R}^d)}\, S_m(t-p)\quad \text{in $L^2(\mathbb{R}^d)$}\,,
\]
holds, for some sampling functions $S_m \in V_\Phi^2$,\, $m=1, 2, \dots,M$. Moreover, the sequence $\{S_m(t-p)\}_{p\in \mathbb{Z}^d;\,m=1, 2, \dots,M}$ is a frame for $V_\Phi^2$.
As a consequence of the RKHS setting the convergence of the series in the $L^2(\mathbb{R}^d)$-norm sense implies pointwise convergence which is absolute and uniform on $\mathbb{R}^d$. As we will see later (see Section \ref{subsection4-4}), this oversampling can be reduced by sampling on a sublattice $P\mathbb{Z}^d$ of 
$\mathbb{Z}^d$, where $P$ denotes a $d\times d$ matrix with integer entries and positive determinant.

\medskip 

$\bullet$ The case where the group $G$ is the {\em semi-direct product of two groups} can be easily reduced to the described situation in Section \ref{subsection4-1}. Suppose that $(k,h)\mapsto U(k,h)$ is a unitary representation of the semi-direct product group $G=K\rtimes_\sigma H$ (or, in particular, the direct product $G=K\times H$) on a separable Hilbert space $\mathcal{H}$, where $K$ is a countable discrete group and $H$ a finite not necessarily abelian group; the subscript $\sigma$ denotes the action of the group $H$ on the group $K$, i.e., a homomorphism $\sigma: H \rightarrow Aut(K)$ mapping $h\mapsto \sigma_h$. The composition law in $G$ is 
$(k_1,h_1)\,(k_2,h_2):=(k_1\sigma_{h_1}(k_2),h_1h_2)$ for $(k_1,h_1), \,(k_2,h_2) \in G$. In general, the group $G=K\rtimes_\sigma H$ is not abelian. In case 
$\sigma_h\equiv Id_K$ for each $h\in H$ we recover the direct product group $G=K\times H$.

Assume that for a fixed $\varphi \in \mathcal{H}$ the sequence $\big\{U(k,h)\varphi\big\}_{(k,h)\in G}$ is a Riesz sequence for $\mathcal{H}$. Thus, the $U$-invariant subspace in 
$\mathcal{H}$ spanned by $\big\{U(k,h)\varphi\big\}_{(k,h)\in G}$ can be described as 
\[
\mathcal{V}_\varphi=\Big\{ \sum_{(k,h)\in G} x(k,h)\,U(k,h)\varphi\,\, :\,\, \{x(k,h)\}_{(k,h)\in G}\in \ell^2(G) \Big\}\,.
\]
Since $U(k,h)\varphi=U[(k,1_H)(0_K,h)]\varphi =U(k,1_H)\varphi_h$, where $\varphi_h:=U(0_K,h)\varphi$ for $h\in H$. Assuming that the order of the group $H$ is  $N$, the subspace $\mathcal{V}_\varphi$ coincides with the subspace $\mathcal{V}_\Phi$ generated by the set $\Phi=\{\varphi_1, \varphi_2, \dots, \varphi_N\}$, i.e.,
\[
\mathcal{V}_\Phi=\Big\{ \sum_{n=1}^N\sum_{k\in K} x_{n}(k) U(k,1_H)\varphi_{n} \, :\,  x_{n}\in \ell^2(N),\,\, n=1, 2, \dots, N\Big\}\,,
\]
where $x_n(k):=x(k,h_n)$ and $\varphi_n:=\varphi_{h_n}$, \, $n=1, 2, \dots, N$. For $M$ fixed elements $\psi_m\in \mathcal{H}$, $m=1,2, \dots ,M$, not necessarily in $\mathcal{V}_\Phi$, we consider for each $f\in \mathcal{V}_\Phi$ its generalized samples defined as
\begin{equation}
\label{samples3}
\mathcal{L}_m f(k):=\big\langle f, U(k, 1_H)\,\psi_m \big\rangle_\mathcal{H}\,,\quad \text{ $k\in K$\,, \quad $m=1,2,\dots, M$}\,.
\end{equation}
Notice that these samples are a particular case of samples \eqref{samples1}.
Then, under the hypotheses in Thm.\ref{Usampling} on the matrix $A=[a_{m,n}]$ where $a_{m,n}(k)=\big\langle \varphi, U[(-k,h_n)^{-1}]\psi_m\big\rangle_\mathcal{H}$, there exist $M$ elements $S_m\in \mathcal{V}_\Phi$ such that the sequence $\big\{U(k,1_H)S_m\big\}_{k\in K;\,m=1,2,\dots, M}$ is a frame for $\mathcal{V}_\Phi$, and for each $f\in \mathcal{V}_\Phi$ we have the reconstruction formula
\begin{equation}
\label{sampling3}
f=\sum_{m=1}^M\sum_{k\in K} \mathcal{L}_m f(k)\,U(k,1_H)S_m \quad \text{ in $\mathcal{H}$}\,.
\end{equation}

\medskip

$\bullet$ An important case of the example above is given by {\em crystallographic groups}. Namely, the Euclidean motion group $E(d)$ is the semi-direct product $\mathbb{R}^d \rtimes_{\sigma} O(d)$ corresponding to the homomorphism $\sigma : O(d) \rightarrow Aut(\mathbb{R}^d)$ given by $\sigma_{\gamma}(x) = \gamma x$, where $\gamma \in O(d)$ and $x\in \mathbb{R}^d$;  $O(d)$ denotes the orthogonal group of order $d$. The composition law on $E(d) = \mathbb{R}^d \rtimes_{\sigma} O(d)$ reads 
$(x, \gamma) \cdot  (x', \gamma') = (x + \gamma x', \gamma \gamma')$.

Let $P$ be a non-singular $d\times d$ matrix and $\Gamma$ a finite subgroup of $O(d)$ of order $N$ such that  $\gamma(P\mathbb{Z}^d)=P\mathbb{Z}^d$ for each $\gamma \in \Gamma$. We consider the {\em crystallographic group} $\mathcal{C}_{P,\Gamma}:=P\mathbb{Z}^d \rtimes_\sigma \Gamma$ and its {\em quasi regular representation} (see Ref.~\cite{barbieri:15}) on $L^2(\mathbb{R}^d)$
\[
U(p,\gamma)f(t)=f[\gamma^{\top}(t-p)]\,,\quad \text{$p\in P\mathbb{Z}^d$, $\gamma\in \Gamma$ and $f\in L^2(\mathbb{R}^d)$}\,.
\]
For a fixed $\varphi \in L^2(\mathbb{R}^d)$ such that the sequence $\big\{ U(p,\gamma)\varphi \big\}_{(p,\gamma)\in \mathcal{C}_{P,\Gamma}}$ is a Riesz sequence for $L^2(\mathbb{R}^d)$ we consider the $U$-invariant subspace in $L^2(\mathbb{R}^d)$
\begin{equation}
\label{cryspace}
\mathcal{V}_\varphi=\Big\{\sum_{(p,\gamma)\in \mathcal{C}_{P,\Gamma}} x(p,\gamma)\, \varphi [\gamma^{\top}(t-p)] \,\,:\,\, \{x(p,\gamma)\}\in \ell^2(\mathcal{C}_{P,\Gamma})\Big\}
\end{equation}
Choosing $M$ functions $\psi_m\in  L^2(\mathbb{R}^d)$, $m=1,2,\dots,M$, we consider the average samples of $f\in \mathcal{V}_\varphi$
\[
\mathcal{L}_m f(p)=\langle f, U(p, I) \psi_m\rangle=\langle f, \psi_m(\cdot-p)\rangle\,,\quad p\in P\mathbb{Z}^d\,.
\]
Denoting $\{\gamma_1=I, \gamma_2, \dots, \gamma_N\}$ the elements of the group $\Gamma$, under the hypotheses of Thm. \ref{Usampling} on the matrix $A=[a_{m,n}]$ where 
$a_{m,n}(p)=\big\langle \varphi(t), \psi_m(\gamma_n t-p)\big\rangle_{L^2(\mathbb{R}^d)}$, there exist $M\geq N$ sampling functions $S_m \in \mathcal{V}_\varphi$ for $m=1,2,\dots,M$, such that the sequence 
$\{S_m(\cdot-p)\}_{p\in P\mathbb{Z}^d;\,m=1,2,\dots,M}$ is a frame for $\mathcal{V}_\varphi$, and the sampling expansion
\begin{equation}
\label{sampling4}
f(t)=\sum_{m=1}^M \sum_{p\in P\mathbb{Z}^d} \big\langle f, \psi_m(\cdot-p)\big\rangle_{L^2(\mathbb{R}^d)} \, S_m(t-p) \quad \text{in $L^2(\mathbb{R}^d)$}
\end{equation}
holds. If the generator $\varphi$ is continuous in $\mathbb{R}^d$ and the function $t\mapsto \sum_{p\in \mathbb{Z}^d}|\varphi(t-p)|^2$ is bounded on $\mathbb{R}^d$, a standard argument shows that 
$\mathcal{V}_\varphi$ is a RKHS of bounded continuous functions in $L^2(\mathbb{R}^d)$ (see, for instance, Ref.~\cite{garcia:19}). As a consequence, convergence in $L^2(\mathbb{R}^d)$-norm implies pointwise convergence which is absolute and uniform on $\mathbb{R}^d$.
\subsection{The case of pointwise samples whenever $\mathcal{H}=L^2(\mathbb{R}^d)$}
\label{subsection4-3}
Assume here that $G$  is a countable discrete subgroup of a locally  compact abelian group $\widetilde{G}$ and let $t\in \widetilde{G} \mapsto U(t)\in \mathcal{U}(L^2(\widetilde{G}))$ be a unitary representation of 
$\widetilde{G}$ on $L^2(\widetilde{G})$. Let $\mathcal{V}_\Phi$ be the corresponding $U$-invariant subspace of 
$\mathcal{H}=L^2(\widetilde{G})$ given  in \eqref{invariantsubspace}; for any $f\in \mathcal{V}_\Phi$ we consider the samples defined in \eqref{samples2} from $M$ fixed points $t_m \in \widetilde{G}$,\, $m=1,2, \dots, M$, i.e., 
\begin{equation}
\label{samples2bis}
\mathcal{L}_{m}f(g):=\big[U(-g)f \big](t_m),\quad g\in G,\,\, m=1,2, \dots, M.
\end{equation}
Let  $A=[a_{m,n}]$ be the $M\times N$ matrix where $a_{m,n}(g)=\big[U(-g)\varphi_n \big](t_m)$,\, $g\in G$; assuming that, for each $t\in \widetilde{G}$, the sequence 
$\{[U(g)\varphi_n](t)\}_{g\in G}$ belongs to $\ell^2(G)$ for each $n=1, 2,\dots, N$, the matrix $A$ has its entries in $\ell^2(G)$. Moreover, if the functions 
$[U(g)\varphi_n](t)$,  $g\in G$ and $n=1,2,\dots, N$, are continuous on $\widetilde{G}$, and the condition
\begin{equation}
\label{bounded}
\sup_{t\in \widetilde{G}} \sum_{g\in G} \big|[U(g)\varphi_n](t) \big|^2 <+\infty\,, \quad n=1,2,\dots, N\,,
\end{equation}
holds, then the subspace $\mathcal{V}_\Phi$ is a reproducing kernel Hilbert space of continuous bounded functions in $L^2(\widetilde{G})$. In fact, it is a necessary and sufficient condition as the following result shows; its proof is analogous to that in \cite[Lemma 4.2]{garcia:19}.
\begin{prop}
For any $\{x_n(g)\}_{gÊ\in G;\,n=1,2,\dots, N} \in\ell^2_{_N}(G)$ the series $$\displaystyle{\sum_{n=1}^N \sum_{g\in G} x_n(g)\, [U(g)\varphi_n](t)}$$ converges pointwise to a  continuous bounded function on $\widetilde{G}$ if and only if for each $g\in G$ and $n=1,2,\dots, N$, the function $U(g)\varphi_n$ is continuous on 
$\widetilde{G}$, and condition \eqref{bounded} holds.
\end{prop}

\medskip

\noindent Notice that, whenever $\mathcal{H}=L^2(\mathbb{R}^d)$ and  $[U(p)f](t):=f(t-p)$,\, $t\in \mathbb{R}^d$\,,\, $p\in \mathbb{Z}^d$, the samples in  \eqref{samples2bis} read
\[
\mathcal{L}_{m}f(p)=\big[U(-p)f \big](t_m)=f(p+t_m)\,, \quad \text{$p\in \mathbb{Z}^d$ \,\text{ and }\, $m=1,2,\dots, M$}\,.
\]

\medskip

$\bullet$ Choosing $\mathcal{H}:=L^2(\mathbb{R}^d)$, $G:=\mathbb{Z}^d$ and $\big(U(p)f\big)(t):=f(t-p)$, $t\in \mathbb{R}^d$ and $p\in \mathbb{Z}^d$. Thus, under hypotheses in 
Thm.~\ref{Usampling} on the matrix $A=[a_{m,n}]$ where $a_{m,n}(p)=\varphi_n(t_m+p)$ we obtain oversampled pointwise sampling in the shift-invariant subspace $V_\Phi^2$ of $L^2(\mathbb{R}^d)$, i.e., for each $f\in V_\Phi^2$ a sampling expansion having the form
\[
f(t)=\sum_{m=1}^M\sum_{p\in \mathbb{Z}^d}f(p+t_m)\, S_m(t-p)\,,\quad t\in \mathbb{R}^d
\]
holds, for some functions $S_m \in V_\Phi^2$,\, $m=1, 2, \dots,M$. The convergence of the series in $L^2(\mathbb{R}^d)$-norm implies pointwise convergence which is absolute and uniform on $\mathbb{R}^d$.

\medskip

$\bullet$ In the case of the quasi regular representation of the crystallographic group $\mathcal{C}_{P,\Gamma}=P\mathbb{Z}^d\rtimes_\sigma \Gamma$, for each $f\in \mathcal{V}_\varphi$ defined in \eqref{cryspace} the samples \eqref{samples2} read
\[
\mathcal{L}_m f(p)=\big[U(-p,I)f\big](t_m)=f(p+t_m)\,, \quad p\in P\mathbb{Z}^d \,\text{ and }\, m=1,2,\dots, M\,.
\]
Under hypotheses in Thm.~\ref{Usampling} on the matrix $A=[a_{m,n}]$ where $a_{m,n}(p)=\varphi[\gamma_n^\top(t_m-p)]$, there exist $M$ functions $S_m\in \mathcal{V}_\varphi$, $m=1,2,\dots, M$, such that for each $f\in \mathcal{V}_\varphi$ the sampling formula
\[
f(t)=\sum_{m=1}^M \sum_{p\in P\mathbb{Z}^d} f(p+t_m)\,S_m(t-p)\,,\quad t\in \mathbb{R}^d
\]
holds. The convergence of the series in the $L^2(\mathbb{R}^d)$-norm sense implies pointwise convergence which is absolute and uniform on $\mathbb{R}^d$.
\subsection{Sampling in a subgroup $H$ of $G$}
\label{subsection4-4}
Let $(G, +)$ be a countable discrete LCA group, and let  $H$ be a  subgroup of $G$ with finite index $L$. We fix a set $\{g_0, g_1, \dots,g_L\}$ of representatives of the cosets of $H$, i.e., the group $G$ can be decomposed as
\[
G= (g_1+H) \cup (g_2+H) \cup \dots \cup (g_{L}+H)\,\, \text{with}\,\, (g_l+H) \cap (g_{l'}+H)=\varnothing \,\,\text{for $l\neq l'$}\,.
\]
Given a unitary representation $g\mapsto U(g)$ of the group $G$ on a separable Hilbert space $\mathcal{H}$ and a set of generators $\Phi=\{\varphi_{1},\varphi_{2,}\ldots,\varphi_{N}\}$ in $\mathcal{H}$, we consider the subspace $\mathcal{V}_{\Phi}=\overline{\espan}_{\mathcal{H}}\{U(g)\varphi_n\}_{g\in G;\,n=1,2,\dots,N}$.
In case $\{U(g)\varphi_n\}_{g\in G;\,n=1,2,\dots,N}$ is a Riesz sequence in $\mathcal{H}$, it can be expressed as
\[
\mathcal{V}_\Phi=\Big\{\sum_{n=1}^N\sum_{g\in G} x_n(g)\,U(g)\varphi_n\, :\,  x_n \in \ell^2(G)\Big\}=
\Big\{\sum_{n=1}^N\sum_{l=1}^L \sum_{h\in H} x_n(g_l+h)\,U(g_l+h)\varphi_n\Big\}\,,
\]
where the sequence
\[
\mathbf{x}(h) :=\big(x_{11}(h),\dots, x_{1L}(h), x_{21}(h),\dots, x_{2L}(h), \dots,x_{N1}(h),\dots, x_{NL}(h) \big)^\top\in \ell^2_{_{NL}}(H)\,,
\]
with $x_{nl}(h):=x_n(g_l+h)$. From now on we consider a new index $nl$, from $11$ to $NL$,  whose order is the indicated above.
Next, for $M$ fixed elements $\psi_{m}\in \mathcal{H}$,\, $m=1, 2, \dots,M$, not necessarily in $\mathcal{V}_\Phi$, for each $f\in \mathcal{V}_\Phi$ we define its generalized samples
\begin{equation}
\label{samples4}
\mathcal{L}_{m}f(h)=\big\langle f, U(h)\psi_{m} \big\rangle_{\mathcal{H}},\quad \text{$h\in H$ and  $m=1,2,\dots, M$}\,.
\end{equation}
For $f=\sum_{n=1}^N\sum_{l=1}^L \sum_{k\in H} x_n(g_l+k)\,U(g_l+k)\varphi_n$ in $\mathcal{V}_\Phi$, the samples \eqref{samples4} can be expressed as
\[
\mathcal{L}_{m}f(h)=\sum_{n=1}^N\sum_{l=1}^L \sum_{k\in H} x_n(g_l+k)\big\langle \varphi_n, U(h-g_l-k)\psi_m\big\rangle=\sum_{n=1}^N\sum_{l=1}^L \big(a_{m,nl}\ast_Hx_{nl}\big)(h)\,,
\]
where $a_{m,nl}(h):=\big\langle \varphi_n, U(h-g_l)\psi_m\big\rangle_\mathcal{H}$,\, $h\in H$, for $m=1, 2, \dots,M$,\, $n=1, 2, \dots,N$ and $l=1, 2, \dots,L$. Notice that each $a_{m,nl} \in \ell^2(H)$. The subscript $*_H$ means convolution over the subgroup $H$.

If we consider the $M\times NL$ matrix $A=[a_{m,nl}]$, the hypotheses in Thm.\ref{Usampling} on matrix $A$ proves, with slight differences, that $M\geq NL$ and there exists a frame sequence $\big\{T_h \mathbf{b}_{m}\big\}_{h\in H;\, m=1,2,\dots,M}$ for 
$\ell^2_{_{NL}}(H)$ which  is a dual frame of $\big\{T_h \mathbf{a}_{m}^*\big\}_{h\in H;\, m=1,2,\dots,M}$. Thus, for any $\mathbf{x} \in \ell^2_{_{NL}}(G)$ we have  
\begin{equation}
\label{expansion2}
\mathbf{x}=\sum_{m=1}^M \sum_{h\in H} \big\langle \mathbf{x}, T_h \mathbf{a}_{m}^* \big \rangle_{\ell^2_{_{NL}}(H)}\, T_h \mathbf{b}_{m}=\sum_{m=1}^M\sum_{h\in H} \mathcal{L}_{m}f(h)\, T_h \mathbf{b}_{m} \quad \text{in $\ell^2_{_{NL}}(H)$}\,.
\end{equation}
Next, the natural isomorphism $\mathcal{T}_{U,\Phi}: \ell^2_{_{NL}}(H) \rightarrow \mathcal{V}_\Phi$ which maps the standard orthonormal basis 
$\{\boldsymbol{\delta}_{h,nl}\}$ for $\ell^2_{_{NL}}(H)$ on the Riesz basis $\big\{U(g_l+h)\varphi_n\big\}$ for $\mathcal{V}_\Phi$, and satisfies the {\em shifting property}
$\mathcal{T}_{U,\Phi}\big(T_h \mathbf{b} \big)=U(h)\big(\mathcal{T}_{U,\Phi}\mathbf{b} \big)$ for each $h\in H$ and $\mathbf{b}\in \ell^2_{_{NL}}(H)$. 

Applying the isomorphism $\mathcal{T}_{U,\Phi}$ in \eqref{expansion2} we obtain that any $f=\mathcal{T}_{U,\Phi} \mathbf{x} \in \mathcal{V}_\Phi$ can be recovered from data $\{\mathcal{L}_{m}f(h)\}_{h\in H;\,m=1,2,\dots,M}$ by means of the sampling formula
\begin{equation}
\label{sampling5}
f=\sum_{m=1}^M\sum_{h\in H} \mathcal{L}_m f(h)\,U(h)S_m \quad \text{ in $\mathcal{H}$}\,,
\end{equation}
for some sampling functions $S_m=\mathcal{T}_{U,\Phi}\mathbf{b}_m \in \mathcal{V}_\Phi$,\, $m=1, 2, \dots, M$. Moreover, the sequence $\big\{U(h)S_m\big\}_{h\in H;\,m=1,2,\dots, M}$ is a frame for $\mathcal{V}_\Phi$.

\medskip

$\bullet$ In particular, consider $\mathcal{H}:=L^2(\mathbb{R}^d)$, $G:=\mathbb{Z}^d$ and $\big[U(p)f\big](t):=f(t-p)$, $t\in \mathbb{R}^d$ and $p\in \mathbb{Z}^d$. Let $P\mathbb{Z}^d$ be a sublattice in $\mathbb{Z}^d$ where $P$ denotes a $d\times d$ matrix of integer entries with positive determinant $L:=\det P$. Under the hypotheses in Thm.~\ref{Usampling} on the matrix $A=[a_{m,nl}]$ where $a_{m,nl}(p)=\big\langle \varphi_n(t), \psi_m(t-p+g_l)\big\rangle_{L^2(\mathbb{R}^d)}$, the sampling formula \eqref{sampling5} gives an average sampling formula in the classical shift-invariant subspace $V_\Phi^2$ of $L^2(\mathbb{R}^d)$, i.e., for each $f\in V_\Phi^2$ formula \eqref{sampling5} reads
\[
f(t)=\sum_{m=1}^M\sum_{p\in P\mathbb{Z}^d} \langle f, \psi_m(\cdot-p)\rangle_{L^2(\mathbb{R}^d)}\, S_m(t-p)\,,\quad t\in \mathbb{R}^d\,,
\]
for some sampling functions $S_m \in V_\Phi^2$,\, $m=1, 2, \dots,M$. Moreover, the sampling sequence $\{S_m(t-p)\}_{p\in P\mathbb{Z}^d;\,m=1, 2, \dots,M}$ is a frame for $V_\Phi^2$.
The convergence of the series in the $L^2(\mathbb{R}^d)$-norm sense implies pointwise convergence which is absolute and uniform on $\mathbb{R}^d$.

\medskip

$\bullet$ Whenever $\mathcal{H}=L^2(\mathbb{R}^d)$, for $M$ fixed points $t_m\in \mathbb{R}^d$, we consider in the shift-invariant subspace $V_\Phi^2$ the samples $\mathcal{L}_{m}f(h):=\big[U(-h)f \big](t_m)$, $h\in H$,\, $m=1,2, \dots, M$, for any $f=\sum_{n=1}^N\sum_{l=1}^L \sum_{k\in H} x_n(g_l+k)\,U(g_l+k)\varphi_n$ in $\mathcal{V}_\Phi$ the samples can be expressed as
\[
\mathcal{L}_{m}f(h)=\sum_{n=1}^N\sum_{l=1}^L \sum_{k\in H} x_n(g_l+k) \big[U(k-h)U(g_l)\varphi_n\big](t_m)=\sum_{n=1}^N\sum_{l=1}^L \big(a_{m,nl}\ast_Hx_{nl}\big)(h)\,,
\]
where $a_{m,nl}(h):=\big[U(-h+g_l)\varphi_n\big](t_m)$,\, $h\in H$, for $m=1, 2, \dots,M$,\, $n=1, 2, \dots,N$ and $l=1, 2, \dots,L$. 

In particular, if we sample any function in  $V_\Phi^2$ on a sublattice $P\mathbb{Z}^d$ of $\mathbb{Z}^d$, under the hypotheses in Thm.\ref{Usampling} on the $M\times NL$ matrix $A=[a_{m,nl}]$ where $a_{m,nl}(p)=\varphi_n(t_m+p-g_l)$, there exist $M\geq NL$ sampling functions $S_m \in V_\Phi^2$,\, $m=1, 2, \dots,M$, such that we recover any $f\in V_\Phi^2$
from the samples $\{\mathcal{L}_{m}f(p)=f(p+t_m)\}_{p\in P\mathbb{Z}^d;\,m=1, 2, \dots,M}$ by means of the sampling formula
\[
f(t)=\sum_{m=1}^M\sum_{p\in P\mathbb{Z}^d} f(p+t_m)\, S_m(t-p)\,,\quad t\in \mathbb{R}^d\,.
\]
Moreover, the sampling sequence $\{S_m(t-p)\}_{p\in P\mathbb{Z}^d;\,m=1, 2, \dots,M}$ is a frame for $V_\Phi^2$.
The convergence of the series in the $L^2(\mathbb{R}^d)$-norm sense implies pointwise convergence which is absolute and uniform on $\mathbb{R}^d$.
\subsection{Some final comments}
\label{subsection4-5}
Our main sampling result, Theorem~\ref{Usampling}, involves some sampling conditions appearing in the mathematical literature; thus, condition $(b)$ says that $\boldsymbol{\mathcal{L}}=(\mathcal{L}_1, \mathcal{L}_2, \dots, \mathcal{L}_M)$ is a {\em stable averaging sampler} for $\mathcal{V}_\Phi$ as it was introduced in Ref.~\cite{aldroubi:05}. Besides, formula \eqref{sampling2} is the expected reconstruction formula in the $U$-invariant subspace $\mathcal{V}_\Phi$. The differences can be found in conditions $(a)$-$(c)$-$(d)$ used here since these conditions are directly related to the filtering process defining the samples \eqref{genalizedsamples}. Consequently, these conditions are given in terms of the convolution system $\mathcal{A}$ and its transfer matrix $\widehat{A}$. In references concerning shift-invariant subspaces these conditions are given in terms of some Gram matrices (see, for instance, Refs.~\cite{aldroubi:96,aldroubi:05}), or in terms of the so called {\em modulation matrix} whose entries are given in terms of the Zak transform as $\big(Z\mathcal{L}_m \varphi_n\big)(0,w)$ (see, for instance, Refs.~\cite{hector:14,garcia:09,shang:07}).

As it was mentioned before, some previous sampling results can be seen as particular examples of this approach. As a non-exhaustive sample of such examples we can cite sampling in shift-invariant subspaces Refs.~\cite{aldroubi:05,bhandari:12,garcia:06,garcia:09,kang:11,shang:07,zhou:99}, and sampling in $U$-invariant subspaces Refs.\cite{hector:14,garcia:15,garcia:17,pohl:12}. Besides, as it was showed in Section \ref{subsection4-2}, the present approach opens new sampling settings: for instance, those related with crystallographic groups involving examples of practical interest.

\bigskip

\noindent{\bf Acknowledgments:}
This work has been supported by the grant MTM2017-84098-P from the Spanish {\em Ministerio de Econom\'{\i}a y Competitividad (MINECO)}.


\end{document}